\documentclass{amsart}

\usepackage{amsmath,amsthm,amsopn,amstext,amscd,amsfonts,amssymb,mathrsfs,mathtools}
\usepackage{dsfont} 
\usepackage{comment}
\usepackage[active]{srcltx}
\usepackage{graphicx, mathdots}

  \DeclareMathOperator{\diag}{diag}

\newtheorem{definition}{\sc Definition}[section]

\newtheorem{theorem}{\sc Theorem}[section]

\newtheorem{eje}{\sc Example }[section]
\newtheorem{coro}{\sc Corollary}[section]

\newtheorem{remark}{\sc Remark }[section]

\usepackage[active]{srcltx}

\newcommand{\re}{\mathbb{R}}

\newcommand{\de}{\mathbb{D}}

\newcommand{\co}{\mathbb{C}}

\makeatletter\renewcommand\theenumi{\@roman\c@enumi}\makeatother

\usepackage{graphicx, epsfig, subfig}
\usepackage{lscape}
\usepackage{rotating}
\usepackage{caption}
\usepackage{multicol}
\usepackage{colortbl}

\newmuskip\pFqskip
\mathchardef\pFcomma=\mathcode`, 

\newcommand*\pFq[5]{%
 \begingroup
 \begingroup\lccode`~=`,
   \lowercase{\endgroup\def~}{\pFcomma\mkern\pFqskip}%
 \mathcode`,=\string"8000
 {}_{#1}F_{#2}\left(\genfrac..{0pt}{}{#3}{#4};#5\right)%
 \endgroup
}

\begin{document}
\title[On monotonicity of zeros of paraorthogonal polynomials]{On monotonicity of zeros of paraorthogonal polynomials on the unit circle}
\author{K. Castillo}

\address{CMUC, Department of Mathematics, University of Coimbra, 3001-501 Coimbra, Portugal}

\email[K. Castillo]{kenier@mat.uc.pt}

\thanks{The author thanks to J. Petronilho for many fruitful and stimulating discussions.}

\date{\today}
\begin{abstract}
  The purpose of this note is to establish, in terms of the primary coefficients in the framework of the tridiagonal theory developed by Delsarte and Genin in the environment of nonnegative definite Toeplitz matrices,  necessary and sufficient conditions for the monotonicity with respect to a real parameter of zeros of paraorthogonal polynomials on the unit circle. It is also provided tractable sufficient conditions and an application example. These polynomials can be regarded as the characteristic polynomials of any matrix similar to an unitary upper Hessenberg matrix with positive subdiagonal elements. 
\end{abstract}
\maketitle

\section{Introduction}

Toward the beginning of the last century, the methods in the study of zeros of polynomials gained an autonomous interest which gave rise to monographs like Dieudonn\'e's {\em La th\'eorie analytique des poly\^omes d'une variable (\`a coefficients quelconques)} in 1938, Marden's {\em The geometry of the zeros of a polynomial in a complex variable} in 1949, or Obrechkoff's {\em Zeros of polynomials} \footnote{Available in Bulgarian (with only 500 copies printed) until the English translation published in 2003. In spite of its language, the original was widely used and frequently quoted.} and {\em Verteilung und Berechnung der Nullstellen reeller Polynome} both in 1963. Nowadays, the developments in Statistical Physics, Random Matrix Theory, Probability, and Combinatorics, among other fields, give this topic a new face that attracts a lot of interest to the subject. For a more recent account on analytic theory of polynomials we refer the reader to the monograph by Rahman and Schmeisser \cite{RS02}. In this framework, the study of zeros of {\em orthogonal polynomials} (hereafter abbreviated by OP) have occupied a privileged place since they exhibit very attractive properties. In particular, the monotonicity  with respect to a real parameter of zeros of {\em orthogonal polynomials on the real line} (hereafter abbreviated by OPRL) have been studied as early as the 1880's, when A. Markov established from the absolutely continuous part of the orthogonality measure sufficient conditions for it  \cite[p. $178$]{M86} (cf. \cite[Theorem $7.1.1$]{I05}). As a consequence of their main result, he deduced  the monotonicity of zeros of Jacobi polynomials \footnote{He also stated the monotonicity of zeros of Gegenbauer polynomials, although the proof is based on an incorrect result \cite[p. $181$]{M86} (cf. \cite[p. $121$]{S75}).}. After A. Markov's work ---without knowing about its existence, since the paper was sent to him by Hermite on January 27, 1887 \cite[Lettre $105$]{SH}---  Stieltjes \footnote{Stieltjes' paper was submitted in 1886. Commonly, in the framework of OP this work is erroneously quote as published in that year. It is probably due to the fact that in all the editions of Szeg\H{o}'s book (cf. \cite{S75}) the reference to Stieltjes' work contains this misprint. However, a careful reader may note that Szeg\H{o} used Stieltjes' Collected Papers in the 1914--1918 edition when consulting this work, and the misprint on the date could come from there.} making use of the linear homogeneous differential equation of second order that Jacobi and Gegenbauer polynomials satisfy deduced \footnote{The key argument of Stieltjes' proof, as he wrote in a letter of January 27, 1887 to Hermite \cite[Lettre $105$]{SH}, is a {\em``\'etroite connexion entre la th\'eorie des \'equations alg\'ebriques et celle des formes quadratiques d\'efinies"}. Indeed, in \cite{S87} he proved that if $\mathrm{H}$ is a real symmetric matrix  with nonpositive off-diagonal elements and $\mathrm{H}>\mathrm{O}$, then $\mathrm{H}^{-1}>\mathrm{O}$. (Following Lax \cite[Chapter $10$]{L07}, positivity of a Hermitian matrix $\mathrm{X}$ is denoted as $\mathrm{X}>\mathrm{O}$ or  $\mathrm{O}<\mathrm{X}$.) Nowadays, $\mathrm{H}$ is known as Stieltjes' matrix.} the monotonicity of zeros for these polynomials \cite[Sections $3$ and $4$]{S87}. In the two aforementioned papers the monotonicity of zeros of Gegenbauer polynomials was used to improve some inequalities  for the zeros of Legendre polynomials given by Bruns \footnote{In a note added in January 1887 at the end of his work, after receiving A. Markov's paper, Stieltjes refers such a result in the following terms: {\em ``L'auteur y d\'eduit d'abord la limitation des racines de l'\'equation $X_n=0$ d\'ej\`a obtenue par M. Bruns, et ensuite il obtient ausii et pour la premi\`ere fois, la limitation plus \'etroite (B).".} In a letter of February 3, 1887 to Hermite \cite[Lettre $106$]{SH}, he also commented it. But it is not true that  A. Markov was the first in to improve Bruns' inequalities due to the errors in his work; it is just to attribute the improvement of this result to Stieltjes. All of the above suggests that Stieltjes did not read A. Markov's arguments or, less likely, did not notice the mistakes in his work.} in 1881 and reworked by Szeg\H{o} in  the 1930's \cite{S35}.

Readers familiar with the literature on OP known that the dramatic difference between OPRL and {\it orthogonal polynomials on the unit circle} (hereafter abbreviated by OPUC) is channeled by {\it paraorthogonal polynomials on the unit circle} (hereafter abbreviated by POPUC). The behavior of their zeros, directly or indirectly, is the main reason by which POPUC have received significant attention over the last years (cf. \cite{K85, AGR86, G86, AGR88, GR90, BE91, DG88, DG91a, DG91b, W93, B93, BH95, KN07, S07, S07b} and references therein). From the theoretical point of view, POPUC answer a problem posed by Tur\'an at the beginning of the $1970$'s \cite[Problem LXVI, p. 60]{T74o}:  {\em ``It is known that the zeros of the $n$th orthogonal polynomial (with respect to a Lebesgue-integral function on an interval) separate the zeros of the $(n+1)$th polynomial. What corresponds to this fact  on the unit circle?"} \footnote{We quote the English translation provided by Sz\"usz \cite[Problem LXVI]{T80}.}.  As far as we can tell, this question was solved accidentally by Delsarte and Genin \cite[Section 5]{DG88} \footnote{These authors never mentioned the connection with the question posed by Tur\'an.} when they were working in linear prediction theory. After that, several authors stated additional properties of zeros of POPUC.  A recent work with refined results on the interlacing of zeros of POPUC and historical comments can be found in \cite{CP1}.  

 It is well known that POPUC can be regarded as the characteristic polynomials of any matrix similar to a unitary upper Hessenberg matrix with positive subdiagonal elements (cf. \cite[Proposition 5]{DG91b}). The purpose of this note is to study the monotonicity with respect to a real parameter of zeros of POPUC as an eigenvalue problem for this class of matrices, using in a consequent manner basic methods of linear algebra.  Our main result establishes, in terms of the primary coefficients in the framework of the so-called {\em tridiagonal theory} developed by Delsarte and Genin in the environment of nonnegative definite Toeplitz matrices, necessary and sufficient conditions (and tractable sufficient conditions) for the monotonicity with respect to a real parameter of eigenvalues of unitary upper Hessenberg matrices with positive subdiagonal elements. Our results can be considered as analogues of those presented by Ismail and Muldoon \cite{IM91} (cf. \cite[Section 7.3]{I05}) concerning the monotonicity of zeros of OPRL. In Section \ref{notation} we set up notation and terminology. In Section \ref{mainresult} our main results are stated and proved, and an application example is presented. Section \ref{secfinal} is devoted to some further results and an example within the broader context of matrices with simple eigenvalues on the unit circle. %

\section{Basic concepts and notations}\label{notation} 
We mainly follow the notation of \cite{S05I, S05II, S11}. Denote by $\de$ the (open) unit disk and by $\mathbb{S}^1$ its boundary $\partial \de$, i.e.,$$\de:=\{z \in \co  \,|\,  |z|<1\}\,\, ,\quad \mathbb{S}^1:=\{z \in \co  \,|\,  |z|=1\}\, .$$
Let $\alpha_j \in \de$ ($j=0,\dots,n-1$) and $\tau_n \in \mathbb{S}^1$. In the next definition and subsequently, $\mathrm{I}$ denotes the identity matrix, whose order is made explicit or may be inferred from the context. Set
\begin{align*}
\Theta_j:=\Theta(\alpha_j), \quad \Theta(\alpha):=\begin{pmatrix}
\overline{\alpha} & \rho \\
\rho & -\alpha
\end{pmatrix}\,, \quad \rho:=\left(1-|\alpha|^2\right)^{1/2} ,
\end{align*}
and 
$$
\mathrm{G}_j:=\diag \left(\mathrm{I}_{j}, \Theta_j, \mathrm{I}_{n-j-1} \right), \quad \mathrm{G}_n:=\diag(\mathrm{I}_{n}, \overline{\tau}_n).
$$
Define the $(n+1)$-by-$(n+1)$ matrix 
\begin{align}\label{lm}
\mathrm{G}:&= \mathrm{G}_0\mathrm{G}_1\cdots \mathrm{G}_n=\begin{pmatrix}
\overline{\alpha}_0 & \rho_0 \overline{\alpha}_1 & \rho_0 \rho_1 \overline{\alpha}_3 & \cdots & \rho_0 \cdots \rho_{n-1} \overline{\tau}_n\\
\rho_0 & -\alpha_0 \overline{\alpha}_1 & -\alpha_0 \rho_1 \overline{\alpha}_2 & \cdots & -\alpha_0 \rho_1 \cdots \rho_{n-1} \overline{\tau}_n\\
 & \rho_1 & -\alpha_1 \overline{\alpha}_2 & \cdots & -\alpha_1 \rho_2 \cdots \rho_{n-1} \overline{\tau}_n\\
  & & \ddots & \ddots & \vdots \\
  & & & \rho_{n-1} & -\alpha_{n-1} \overline{\tau}_n
\end{pmatrix}.
\end{align}
 By construction, the matrix $\mathrm{G}$ is a unitary upper Hessenberg matrix with positive subdiagonal elements. Conversely, any unitary $(n+1)$-by-$(n+1)$ upper Hessenberg matrix with positive subdiagonal elements is uniquely parameterized in the form \eqref{lm} by $2n+1$ real numbers that compose the parameters of the array $(\alpha_0, \dots, \alpha_{n-1}, \tau_n)$ \cite{G82o} (cf. \cite{G82} and \cite[Proposition $1$]{AG92}). In order to make the notation more transparent, we write $\mathrm{G}(\alpha_0,\dots,$ $\alpha_{n-1},\tau_n)$ instead of $\mathrm{G}$.

 \begin{definition}[cf. {\cite[Proposition $5$]{DG91b}}]\label{POPUC}
Let $\mathrm{G}(\alpha_0,\dots,\alpha_{n-1},\tau_{n})$ be the matrix given by \eqref{lm}, where $\alpha_j \in \de$ ($j=0,\dots,n-1$) and $\tau_n \in \mathbb{S}^1$. The (monic) polynomial $P_{n+1}$ defined by
 \begin{align*}
P_{n+1}(z):=\det \big(z \mathrm{I}-\mathrm{G}(\alpha_0,\dots,\alpha_{n-1},\tau_{n})\big)
 \end{align*}
is the POPUC of degree $n+1$ associated with the array $(\alpha_0,\dots,\alpha_{n-1},\tau_n)$. 
\end{definition}

\begin{definition}[cf. {\cite[Equation $2.29$]{DG91a}}]
 Let $\alpha_j \in \de$ ($j=0,\dots,n-1$) and $\tau_n \in \mathbb{S}^1$. For each $\zeta \in \mathbb{S}^1$, the numbers defined recursively by
 \begin{align}\label{taus}
 \tau_{n}(\zeta):=\tau_n, \quad
 \tau_{j}(\zeta):=\frac{\overline{\zeta} \ \alpha_j+ \tau_{j+1}(\zeta)}{\overline{\alpha}_j \tau_{j+1}(\zeta) +\overline{\zeta}} \, \quad  (j=n-1,\dots,0),
 \end{align}
 are the pseudoreflexion coefficients associated with the array $(\alpha_0,\dots,\alpha_{n-1},\tau_n)$.
 \end{definition}
 
 \section{Main results}\label{mainresult}
In this section we formulate and prove our main results. Our ideas borrow from some ideas of Fan related with the generalized Cayley transform of strictly dissipative matrices and others by Delsarte and Genin in the framework of the tridiagonal theory (cf. \cite{DG88, DG91a, DG91b, BH95}), where a one-parameter second order recurrence relation is the main object of study.

\begin{theorem}\label{main}
Let $\mathrm{G}(\alpha_0,\dots,\alpha_{n-1},\tau_{n})$ be a differentiable matrix-valued function of the real variable $t$ given by \eqref{lm}, where for all $t$, $\alpha_j:=\alpha_j(t) \in \de$ $(j=0,\dots,n-1)$ and $\tau_n:=\tau_n(t) \in \mathbb{S}^1$. Define $\tau_j(\zeta)$ via \eqref{taus} for $\zeta \in \mathbb{S}^1\setminus S$, $S$ being the set of eigenvalues of  $\mathrm{G}(\alpha_0,\dots,\alpha_{n-1},\tau_n)$. Define recursively the numbers \footnote{Cf. \cite[Equations $3.14$ and $3.19$]{DG91a}. Here and subsequently, $\zeta^{1/2}$ denotes either of the square roots of $\zeta$.}
 \begin{align*}
 \beta_0(\zeta)&:=\frac{1}{\zeta^{1/2}+\tau_0(\zeta) \zeta^{1/2}},\\
  \beta_j(\zeta)&:= \frac{1}{\beta_{j-1}(\zeta)}\frac{1}{\overline{\zeta}+ \tau_j(\zeta)\overline{\alpha_{j-1}}}\frac{1}{1-\overline{\tau_j} \alpha_{j+1}} \quad (j=1,\dots,n),
 \end{align*}
and the polynomials \footnote{Cf. \cite[Equation $3.7$]{DG91a}.}
 \begin{align}\label{rec1}
 \nonumber p_{-1}(z,\zeta)&:=0, \quad p_0(z,\zeta):=p_0 \in \mathbb{R}\setminus\{0\},\\
  p_{j+1}(z, \zeta)&:=\left(\beta_j(\zeta)+\overline{\beta_j(\zeta)} z\right)p_j(z,\zeta)-zp_{j-1}(z,\zeta) \quad (j=0,\dots,n).
\end{align}
Define also the matrix
\begin{align*}
\mathrm{A}:=\zeta^{1/2}
\begin{pmatrix}
\overline{\beta_0(\zeta)} & & & &\\
-1 &  \overline{\beta_1(\zeta)}& & &\\
& \ddots & \ddots &  &\\
& & -1 &\overline{\beta_n(\zeta)}&  
\end{pmatrix},
\end{align*}
with Toeplitz decomposition $\mathrm{H}+i\mathrm{K}$ \footnote{I.e., $\mathrm{H}:=(\mathrm{A}+\mathrm{A}^*)/2$ and $\mathrm{K}:=(\mathrm{A}-\mathrm{A}^*)/(2i)$.}. Then $\mathrm{H}>\mathrm{O}$ \footnote{I.e., $\mathrm{A}$ is strictly dissipative.}, the eigenvalues of $\mathrm{G}(\alpha_0,\dots,$ $\alpha_{n-1},\tau_{n})$ coincide with the zeros of $p_{n+1}(z, \zeta)$, and for all $\eta  \in S$,
\begin{align}\label{eq}
\left(\mathrm{p}, \mathrm{H}^{1/2}\frac{\mathrm{d}}{\mathrm{d} t}\big(\mathrm{H}^{-1/2}\mathrm{K}\mathrm{H}^{-1/2})\mathrm{H}^{1/2}\mathrm{p}\right)=\frac{(\mathrm{p},\mathrm{H}\mathrm{p})}{\cos(\overline{\zeta}\eta)-1}\, \frac{\mathrm{d}}{\mathrm{d} t}\arg(\eta)\quad ({\rm mod}(0,2\pi]),
\end{align}
where $\mathrm{p}:=(p_0(\eta, \zeta), \dots, p_n(\eta, \zeta))^{\mathrm{T}}$.
\end{theorem}

\begin{proof}
The fact that $\mathrm{H}>\mathrm{O}$ is well known; two different proofs of this fact are given in \cite[Section $3$]{DG91a} and \cite[pp. 435-436]{DG91b}.
Since $\mathrm{G}(\alpha_0,\dots,\alpha_{n-1},\tau_n)$ is a differentiable matrix-valued function having simple eigenvalues (cf. \cite[Proposition $3.26$]{S10}), then its eigenvalues depend differentiably on $t$  (cf. \cite[Theorem 7, p. 130]{L07}). Furthermore, for each eigenvalue of  $\mathrm{G}(\alpha_0,\dots,\alpha_{n-1},\tau_n)$, we can choose an eigenvector that depends differentiably on $t$ (cf. \cite[Theorem 8, p. 130]{L07}). 

Let $\mathrm{G}(\alpha_0,\dots,\alpha_{n-1},$ $\tau_n)$ be  partitioned as
\begin{align}\label{G}
\mathrm{G}(\alpha_0,\dots,\alpha_{n-1},\tau_{n})=\begin{pmatrix}
\mathrm{G}_{11} & \mathrm{G}_{12}\\
\mathrm{G}_{21} & \mathrm{G}_{22}
\end{pmatrix},
\end{align}
$\mathrm{G}_{11}$ being the $(j+1)$-by-$(j+1)$ leading principal submatrix of $\mathrm{G}(\alpha_0,\dots,$ $\alpha_{n-1},\tau_n)$. We claim that $\mathrm{G}_{22}$ has no eigenvalues on $\mathbb{S}^1$. Indeed, it can be easily seen that $\mathrm{G}_{22}$ is the $(n-j)$-by-$(n-j)$ trailing principal submatrix of each of the matrices $\mathrm{G}(\alpha_j,\dots,\alpha_{n-1},\tau_n)$ and $\mathrm{G}(\alpha_j,\dots,\alpha_{n-1},\tau_n) \, \mathrm{D}$, $\mathrm{D}$ being the diagonal matrix obtained from the identity matrix by replacing the $(1,1)$ entry with a number in $\mathbb{S}^1\setminus \{1\}$. Suppose the assertion is false, and note that the eigenvalues of  $\mathrm{G}(\alpha_0,\dots,\alpha_{n-1},\tau_n)$ are simple and all its eigenvectors have nonzero components \footnote{In order to apply \cite[Lemma $2.2$]{CP1}, it would suffice to note that all the eigenvectors of any normal upper Hessenberg matrix with positive subdiagonal elements have nonzero component at the first (and last) entry (cf. \cite[Lemma 2.1]{L95}).} (cf. \cite[Proposition 5]{DG91b}).  Since $ \mathrm{G}(\alpha_j,\dots,$ $\alpha_{n-1},\tau_{n})$ and  $\mathrm{G}(\alpha_j,\dots,\alpha_{n-1},\tau_{n})  \, \mathrm{D}$ are unitary matrices, these matrices share all the eigenvalues of $\mathrm{G}_{22}$ on $\mathbb{S}^1$. This contradicts the fact that $ \mathrm{G}(\alpha_j,\dots,$ $\alpha_{n-1},\tau_{n})$ and  $\mathrm{G}(\alpha_j,\dots,\alpha_{n-1},\tau_{n})  \, \mathrm{D}$ have no common eigenvalues in agreement with \cite[Lemma $2.2$]{CP1}, and the claim is proved. Consequently, from the equality (cf. \cite[Equation $9$]{BH95}) 
 $$
\mathrm{G}_{11}-\mathrm{G}_{12}(\mathrm{G}_{22}-\zeta \mathrm{I})^{-1}\mathrm{G}_{21}=
 \mathrm{G}(\alpha_0,\dots,\alpha_{j-1},\tau_j(\zeta)),
 $$
 the polynomials
\begin{align*}
 P_1(z, \zeta)&:=z-\overline{\tau_1(\zeta)},\\
P_{j+1}(z,\zeta)&:=\det \big(z \mathrm{I}-\mathrm{G}(\alpha_0,\dots,\alpha_{j-1},\tau_j(\zeta))\big) \quad (j=2,\dots,n),
\end{align*}
are well defined for each $\zeta \in \mathbb{S}^1$. The technical advantage of these polynomials is that any three of them are connected by a simple relation (cf. \cite[Equation $3.3$]{DG91a} and \cite[Equation $10$]{BH95}). After an appropriated normalization (cf. \cite[p. 226]{DG91a}), the resulting polynomials satisfy the recurrence relation \eqref{rec1},  and the second statement of the theorem follows \footnote{One can be also prove this directly using \cite[Equation $4.22$]{DG91b}. But our main interest is in the previous construction.}. Without loss of generality (since, by hypothesis,  $\zeta \in \mathbb{S}^1\setminus S$ \footnote{This means that we are considering only the regular case of the theory presented in \cite{DG91b}.}), we assume the same initial conditions \cite[Equation $3.7$]{DG91a} (cf. \cite[Proposition 3]{DG91a} and \cite[p. 1049]{BH95}). Putting \eqref{rec1} in matrix form, we have (cf. \cite[Equation $2.20$]{DG91b})
$$
(\zeta \mathrm{A}^*+z \mathrm{A})\mathrm{p}=\zeta^{1/2}p_{n+1}(z,\zeta)\mathrm{e}_{n+1},
$$
where $\mathrm{e}_{n+1}:=(0,\dots,0,1)$. Hence
$
- \mathrm{A}^{-1}\mathrm{A}^*\mathrm{p}=\overline{\zeta}\eta\mathrm{p},
$
$\mathrm{A}$ being invertible by definition \footnote{$\mathrm{A}^{-1}\mathrm{A}^*$ is known as the generalized Cayley transform of $\mathrm{A}$ (cf. \cite{F72}).}. It is known (cf. \cite[Corollary $1.1$]{CP1}) that $p_{n}(z, \zeta)$ and $p_{n+1}(z, \zeta)$ have $\zeta$ as the only possible common eigenvalue, then $\mathrm{p}\not=0$ and, consequently, $\mathrm{p}$ is a right eigenvector of $-\mathrm{A}^{-1}\mathrm{A}^*$ associated with the eigenvalue $\overline{\zeta}\eta$. Set $\mathrm{B}:=i\mathrm{A}$. Since $\zeta$ is not an eigenvalue of $\mathrm{G}(\alpha_0,\dots,\alpha_{n-1},\tau_n)$, $\overline{\zeta}\mathrm{G}(\alpha_0,\dots,$ $\alpha_{n-1},$ $\tau_n)$ and $\mathrm{B}^{-1}\mathrm{B}^{*}$ have the same eigenvalues and none of them is equal to $1$.

By \cite[Theorem $2.1$]{F74}, $\mathrm{H}^{1/2}\mathrm{B}^{-1}\mathrm{B}^*\mathrm{H}^{-1/2}$ (since $\mathrm{H}>\mathrm{O}$, $\mathrm{H}^{1/2}$ is well defined and invertible) is unitary and, in turn, the Cayley transform of  $-\mathrm{H}^{-1/2}\mathrm{K}\mathrm{H}^{-1/2}$, i.e.,
$$
-\mathrm{H}^{-1/2}\mathrm{K}\mathrm{H}^{-1/2}=i (\mathrm{I}-\mathrm{H}^{1/2}\mathrm{B}^{-1}\mathrm{B}^*\mathrm{H}^{-1/2})^{-1} (\mathrm{I}+\mathrm{H}^{1/2}\mathrm{B}^{-1}\mathrm{B}^*\mathrm{H}^{-1/2}).
$$
Note that $(\overline{\zeta}\eta,$ $ \mathrm{H}^{1/2}\mathrm{p})$ is an eigenpair of $\mathrm{H}^{1/2}\mathrm{B}^{-1}\mathrm{B}^*\mathrm{H}^{-1/2}$. Hence $\big(i(1-\overline{\zeta}\eta)^{-1}(1+\overline{\zeta}\eta),\mathrm{H}^{1/2}\mathrm{p}\big)$  is an eigenpair of $-\mathrm{H}^{-1/2}\mathrm{K}\mathrm{H}^{-1/2}$, i.e., \footnote{Actually, $
\cot(\arg(\overline{\zeta}\eta)/2)=(\mathrm{p},\mathrm{K}\mathrm{p})/(\mathrm{p},\mathrm{H}\mathrm{p}).
$}
\begin{align}\label{eig}
\Big(\mathrm{H}^{1/2}\mathrm{p}, \big(\mathrm{H}^{-1/2}\mathrm{K}\mathrm{H}^{-1/2}\big)\mathrm{H}^{1/2}\mathrm{p}\Big)=\cot(\arg(\overline{\zeta}\eta)/2)(\mathrm{H}^{1/2}\mathrm{p},\mathrm{H}^{1/2}\mathrm{p})\, .
\end{align}
Finally, since the eigenvalues of $\mathrm{G}(\alpha_0,\dots,$ $\alpha_{n-1},$ $\tau_n)$ are simple, so are those of $\mathrm{H}^{-1/2}\mathrm{K}$ $\mathrm{H}^{-1/2}$, and \eqref{eq} follows by differentiation of \eqref{eig}.
\end{proof}

\begin{remark}\label{remarkrec}
If the starting point in Theorem \ref{main} is the sequence of polynomials defined recursively by \eqref{rec1}, $\beta_j$'s being differentiable functions of the real variable $t$, then \eqref{eq} remains true for all $\zeta \in \mathbb{S}^1$ other than the zeros of the polynomial $p_{n+1}$,  provided the condition $\mathrm{H}>\mathrm{O}$ holds \footnote{Note that this implies that the $\beta_j$'s are nonzero, i.e., $\mathrm{A}$ is invertible.}.
\begin{proof}
Since $\mathrm{A}$ is an invertible and differentiable matrix-valued function of the real variable $t$, so is $\mathrm{A}^{-1}\mathrm{A}^*$. Since this matrix has simple eigenvalues \cite[Section $2.1$]{F72}, it follows that its eigenvalues depend differentiably on $t$ and, for each eigenvalue, we can choose an eigenvector that depends differentiably on $t$. The rest of the proof runs as in the proof of Theorem \ref{main}.
\end{proof}

\end{remark}

The first goal of this work is a direct consequence of Theorem \ref{main}, and it reads as follows:

\begin{coro}\label{remark1}
Assume the notation and conditions of Theorem \ref{main}. Then $\eta$ moves strictly counterclockwise  along $\mathbb{S}^1$ as $t$ increases if and only if
$$
\left(\mathrm{p}, \mathrm{H}^{1/2}\frac{\mathrm{d}}{\mathrm{d} t}\big(\mathrm{H}^{-1/2}\mathrm{K}\mathrm{H}^{-1/2})\mathrm{H}^{1/2}\mathrm{p}\right)<0.
$$
\end{coro}

\begin{remark}\label{remark}
Assume the notation and conditions of Theorem \ref{main}. Define $\mathrm{R}:=\mathrm{H}^{1/2}$ and $\mathrm{L}:=(\mathrm{d}/\mathrm{d}t) \mathrm{R}\mathrm{R}^{-1}\mathrm{K}$; by differentiation, we obtain
$$
 \mathrm{H}^{1/2}\frac{\mathrm{d}}{\mathrm{d} t}\big(\mathrm{H}^{-1/2}\mathrm{K}\mathrm{H}^{-1/2})\mathrm{H}^{1/2}=\frac{\mathrm{d}}{\mathrm{d} t}\mathrm{K}-(\mathrm{L}+\mathrm{L}^T).
$$
This implies that $\eta$ moves strictly counterclockwise  along $\mathbb{S}^1$ as $t$ increases if 
\begin{align}\label{lyap}
\frac{\mathrm{d}}{\mathrm{d} t}\mathrm{K}<\mathrm{L}+\mathrm{L}^T
\end{align}
The inequality \eqref{lyap} may be true even when some of the involved matrices are indefinite. Unfortunately, this (sufficient) condition is unwieldy to work with.\end{remark}

Although we do not pretend to offer a wider study of consequences of the above results, a simple analysis gives tractable sufficient conditions.
\begin{theorem}\label{coromain}
Assume the notation and conditions of Theorem \ref{main}. Define the sets $K_{+}:=\{ t \in \re \,|\, \mathrm{K}>\mathrm{O}\}$ and $K_{-}:=\{t  \in \re \,|\, \mathrm{K}<\mathrm{O}\}$. Set
$$
x_j:=\frac{\mathrm{d}}{\mathrm{d} t}\Re(\zeta^{-1/2}\beta_j(\zeta)), \quad y_j:=\frac{\mathrm{d}}{\mathrm{d} t}\Im(\zeta^{-1/2}\beta_j(\zeta)),
$$
and define also the sets
\begin{align*}
I_{++}&:=\{t\in \re \,|\, (\forall j\in \{0,\dots,n \})[x_j>0 \wedge y_j>0]\},\\
I_{-+}&:=\{t\in \re \,|\, (\forall j\in \{0,\dots,n \})[x_j<0 \wedge y_j>0]\},\\
 I_{--}&:=\{t\in \re \,|\, (\forall j\in \{0,\dots,n \})[x_j<0 \wedge y_j<0]\},\\ 
 I_{+-}&:=\{t\in \re \,|\, (\forall j\in \{0,\dots,n \})[x_j>0 \wedge y_j<0]\}.
 \end{align*}
Then the eigenvalues of $\mathrm{G}(\alpha_0,\dots,\alpha_{n-1},\tau_n)$ move strictly counterclockwise (respectively, clockwise) along $\mathbb{S}^1$ as $t$ increases on each of the nondegenerate intervals \footnote{We are considering the empty set and the singletons as degenerate intervals.}  that make up the set $(I_{++} \cap K_+) \cup (I_{-+} \cap K_-)$ (respectively, $(I_{--}\cap K_+)\cup (I_{+-} \cap K_-)$),  provided that at least one of them exists.
\end{theorem}

\begin{proof}
We only prove the result concerning to the set $I_{++} \cap K_+$; the rest follows in the same way. Assume the notation of Remark \ref{remark}. Denote by $\eta_j$ $(j=0,\dots,n)$ the elements of $C$.  By the proof of Theorem \ref{main}, $\overline{\zeta}\mathrm{G}(\alpha_0,\dots,\alpha_{n-1},\tau_n)$ and  $\mathrm{B}^{-1}\mathrm{B}^*$ have the same eigenvalues and none of them equal to $1$, where $\mathrm{B}$ is as defined there. Consider two different points, say $t_0$ and $t_1$, $t_0<t_1$, in one of the nondegenerate intervals that make up the set $(I_{++} \cap K_+)$, provided it exists. Make temporary explicit that  $\eta_j$, $\mathrm{H}$, and $\mathrm{K}$ depend on $t$. Under our assumptions, we can assert that  $\mathrm{K}(t_1)<\mathrm{K}(t_0)$ and $\mathrm{O}<\mathrm{H}(t_1)<\mathrm{H}(t_0)$. Hence, for any nonzero vector $\mathrm{x}\in \co^{n+1}$,
$$
\frac{(\mathrm{x},-\mathrm{K}(t_0)\mathrm{x})}{(\mathrm{x},\mathrm{H}(t_0)\mathrm{x})}<\frac{(\mathrm{x},-\mathrm{K}(t_1)\mathrm{x})}{(\mathrm{x},\mathrm{H}(t_1)\mathrm{x})}.
$$
Denote the argument of the eigenvalues of $\mathrm{B}^{-1}\mathrm{B}^*$ (which, in turn, are given by $\overline{\zeta}\eta_j(t)$), arranged in decreasing order, by $0<\theta_{n}(t)<\dots<\theta_{2}(t)<\theta_{0}(t)<2\pi$. By Fan's eigenvalues comparison theorem for the generalized Cayley transform \cite[Theorem 6.1]{F74}, we have $\theta_j(t_0)<\theta_j(t_1)$, and the result follows.
\end{proof}

It is worth pointing out that a refined version of \cite[Theorem B]{C15} can be obtained in a straightforward way from Theorem \ref{coromain}.

\begin{coro}\label{c1}
 Assume the notation and conditions of Theorem \ref{coromain}. Assume further that $\mathrm{H}$ does not depend on $t$.  Define the set
\begin{align*}
 I_{0+}&:=\{t \in \re \,|\, (\forall j\in \{0,\dots,n \})[ y_j>0]\},\\ 
 I_{0-}&:=\{t \in \re \,|\, (\forall j\in \{0,\dots,n \})[ y_j<0]\}.
 \end{align*}
Then the eigenvalues of $\mathrm{G}(\alpha_0,\dots,\alpha_{n-1},\tau_n)$ move strictly counterclockwise (respectively, clockwise) along $\mathbb{S}^1$ as $t$ increases on each of the nondegenerate intervals that make up the set $I_{0+}$ (respectively, $I_{0-}$), provided that at least one of them exists.\end{coro}

For completeness, we also indicate the following case:

\begin{coro}
 Assume the notation and conditions of Theorem \ref{coromain}. Assume further that $\mathrm{K}$ does not depend on $t$. Define the sets
 \begin{align*}
 I_{+0}&:=\{t \in \re \,|\, (\forall j\in \{0,\dots,n \})[ x_j>0]\},\\
  I_{-0}&:=\{t \in \re \,|\, (\forall j\in \{0,\dots,n \})[ x_j<0]\}.
\end{align*}
Then the eigenvalues of $\mathrm{G}(\alpha_0,\dots,\alpha_{n-1},\tau_n)$ move strictly counterclockwise (respectively, clockwise) along $\mathbb{S}^1$ as $t$ increases on each of the nondegenerate intervals  that make up the set $I_{+0} \cap K_+$ (respectively, $I_{-0} \cap K_-$),  provided that at least one of them exists.
\end{coro}

The preceding results may seem at first difficult to apply for expeditiously deriving interesting concrete results. However, this is not always true. For instance, when we deal with hypergeometric (or $q$-hypergeometric) polynomials with simple zeros on $\mathbb{S}^1$, the coefficients $\beta_j$'s in the notation of Theorem \ref{main} can be easily determined, as the following example illustrates. Indeed, by virtue of the contiguous relations of hypergeometric functions this class of polynomials furnish an inexhaustible reservoir of examples. 

\begin{eje}\label{eje1}
In \cite[Theorem $1.2$]{DR13}, it is studied, with respect to the parameter $b$, the monotonicity of zeros of the hypergeometric polynomial \footnote{In \cite[Theorem $1.2$]{DR13} the result was stated only for $a>1/2$. Obviously, in such work the polynomial \eqref{q1} is a POPUC by definition (cf. \cite[Equation $1.1$]{DR13}). Consequently, the first statement of \cite[Theorem $1.1$]{DR13} is immediate.  The recurrence relation considered in \cite[Theorem $3.2$]{DR13} can be transform into the simplest form \eqref{rec1} by a natural normalization process (cf. \cite[pp. 226-227]{DG91a}). Therefore the first statements of \cite[Theorem $3.2$]{DR13} and  \cite[Theorem $3.1$]{DR13} are immediate. Note also that in all the above mentioned results, we may say much more about the zeros of the involved polynomials simply because they are POPUC  (see e.g. \cite{CP1} and references therein).}
\begin{align}\label{q1}
r_{n+1}(z):=\frac{(2a)_{n+1}}{(a)_{n+1}}\, \pFq{2}{1}{-n-1,a+bi}{2a}{1-z},
\end{align}
$a$ and $b$ being real numbers with $a$ positive. However, the proof given therein is quite technical and long. In contrast,  Corollary \ref{c1} leads to the result immediately and by very simple means. In order to calculate the coefficients $\beta_j$'s in the notation of Theorem \ref{main}, it is useful to note, using contiguous relations of hypergeometric functions, that the polynomials $r_{j+1}$ ($j=0,\dots, n$)  can be generated recursively by   
\begin{align}\label{rec2}
r_{j+1}(z)=\big((a+j-ib)+(a+j+ib) z\big)r_{j}(z)-j(2a+j-1) z r_{j-1}(z),
\end{align}
with initial conditions $r_{-1}(z):=0$ and $r_0(z):=1$. To achieve our objective we have still to note that \eqref{rec2} can be transformed into the simplest form \eqref{rec1}. Indeed, there exist positive numbers $\lambda_{j}$ \footnote{These numbers are referred in this framework as the Jacobi parameters of the problem. In this case, routine calculations immediately reveal that they are determined uniquely by $j(2a+j-1)\lambda_j\lambda_{j+1}=1$ $(j=1,\dots, n)$ for any nonzero choice of the initial value $\lambda_1$.} (depending only on $a$) and nonzero complex numbers $c_j$ such that $\beta_j=(j+a-ib)\lambda_{j+1}$ and $p_{j+1}=c_{j+1} r_{j+1}$. (Consequently, these polynomials are POPUC.) Clearly, each $\beta_j$  is a nonzero differentiable function of $b$. Note that
$$
\left(\frac{j(2a+j-1)}{(a+j)(a+j+1)}\right)_{j=0}^\infty
$$
is a positive chain sequence associated with the ultraspherical polynomials (cf. \cite[p. 758]{L94}). Wall-Wetzel's theorem (cf. \cite[Theorem $7.2.1$]{I05}) now shows that $\mathrm{H}>\mathrm{O}$, being $\mathrm{H}$ defined as in Theorem \ref{main}, i.e.,
$$
\mathrm{H}=
\begin{pmatrix}
 a \lambda_1 & -1& & &\\
-1 &  (1+a)\lambda_2& \ddots& &\\
& \ddots & \ddots & -1  &\\
& & -1 & (n+a) \lambda_{n+1}&  
\end{pmatrix}.
$$
Since the polynomial \eqref{q1} has no zeros at $z=1$, in agreement with Remark \ref{remarkrec}, there is no loss of generality in assuming $\zeta:=1$. Thus $(\mathrm{d}/\mathrm{d} t)\Im(\beta_j)=-\lambda_j<0$ and, by Corollary \ref{c1}, the zeros of the polynomial \eqref{q1} move strictly clockwise along $\mathbb{S}^1$ as $b$ increases. 
\end{eje}

\section{Further results}\label{secfinal}
The notation of this section differs from that of Section \ref{mainresult}.  As we have already mentioned in the proof of Theorem \ref{main}, for a matrix $\mathrm{U}$ with simple eigenvalues on $\mathbb{S}^1\setminus \{1\}$, there exists a (nonunique) strictly dissipative matrix  $\mathrm{A}$ such that  $\mathrm{U}=\mathrm{A}^{-1}\mathrm{A}^{*}$. Let $\mathrm{H}+i\mathrm{K}$ be the Toeplitz decomposition of $\mathrm{A}$. To study the monotonicity with respect to a real parameter, say  $t$, of the eigenvalues of $\mathrm{U}$,  it suffices to study only the ``sign" of $\mathrm{H}$, $(\mathrm{d}/\mathrm{d} t)\mathrm{H}$, and $(\mathrm{d}/\mathrm{d} t)\mathrm{K}$, as follows from the proof of Theorem \ref{coromain}. For a given $\mathrm{U}$, this requires to identify $\mathrm{H}$ and $\mathrm{K}$, which in itself is not a simple question. Fortunately, when we deal with unitary upper Hessenberg matrices with positive subdiagonal elements, this question is solved by means of the tridiagonal theory and the matrices $\mathrm{H}$ and $\mathrm{K}$ have a simple structure, as already seen. But if the unitary upper Hessenberg matrix considered in the preceding theory is replaced by an arbitrary matrix with simple eigenvalues on $\mathbb{S}^1$, then virtually all of the results remain true, {\em mutantis mutandis}. In view of the above observations, let us (at least) rewrite Theorem \ref{coromain} in the following terms:

\begin{theorem}\label{final}
Let $\mathrm{A}$ be a differentiable matrix-valued function of the real variable $t$. Assume that $\mathrm{A}$ is strictly dissipative with Toeplitz decomposition $\mathrm{H}+i\mathrm{K}$. Define the sets $H_{+}:=\{t\in \re  \,|\, \mathrm{H}>\mathrm{O}\}$, $H_{-}:=\{t\in \re  \,|\, \mathrm{H}<\mathrm{O}\}$,
\begin{align*}
I_{++}&:=\{t\in \re  \,|\,  (\mathrm{d}/\mathrm{d} t)\mathrm{H}>\mathrm{O} \wedge  (\mathrm{d}/\mathrm{d} t)\mathrm{K}>\mathrm{O}\},\\
I_{-+}&:=\{t\in \re  \,|\, (\mathrm{d}/\mathrm{d} t)\mathrm{H}<\mathrm{O} \wedge  (\mathrm{d}/\mathrm{d} t)\mathrm{K}>\mathrm{O}\},\\
 I_{--}&:=\{ t\in \re  \,|\, (\mathrm{d}/\mathrm{d} t)\mathrm{H}<\mathrm{O} \wedge (\mathrm{d}/\mathrm{d} t) \mathrm{K}<\mathrm{O}\},\\
  I_{+-}&:=\{t\in \re  \,|\, (\mathrm{d}/\mathrm{d} t)\mathrm{H}>\mathrm{O} \wedge  (\mathrm{d}/\mathrm{d} t)\mathrm{K}<\mathrm{O}\}.
 \end{align*}
Then the eigenvalues of $\mathrm{A}^{-1}\mathrm{A}^*$ move strictly counterclockwise (respectively, clockwise) along $\mathbb{S}^1$ as $t$ increases on each of the nondegenerate intervals that make up the set $(I_{++} \cap H_-) \cup (I_{+-} \cap H_+)$ (respectively, $(I_{--}\cap H_-) \cup (I_{-+} \cap H_+)$),  provided that at least one of them exists.
\end{theorem}
\begin{proof}
Combining the proofs of Remark \ref{remark} and Theorem \ref{coromain} gives the desired conclusion. 
\end{proof}

The following numerical example helps to elucidate Theorem \ref{final}, and thus Theorem \ref{coromain}, something that the technical simplicity of Example \ref{eje1} does not allow.

\begin{eje}

Define $\mathrm{A}:=\mathrm{H}+i\mathrm{K}$, where we have set 
\begin{align*}
\mathrm{H}&:=
\begin{pmatrix}
 t & 1& & &\\
1 &  t& \ddots& &\\
& \ddots & \ddots & 1  &\\
& & 1 & t&  
\end{pmatrix},\\
\mathrm{K}&:=
\begin{pmatrix}
\pi + \cos t & 1& & &\\
1 &  \pi+\frac{1}{2} \cos t& \ddots& &\\
& \ddots & \ddots &  1&\\
& & 1 &\pi + \frac{1}{n} \cos t&  
\end{pmatrix},
\end{align*}
$t$ being real. For all $t$, $\mathrm{K}$ is strictly diagonally dominant, so $\mathrm{A}$ is strictly dissipative (cf. \cite[Theorem 6.1.10]{HJ}). Using the formula for the eigenvalues of a tridiagonal Toeplitz matrix (cf. \cite[Problem 1.4.P17]{HJ}), we see at once that
\begin{align*}
 H_{-}&=\left\{t\in \re  \,\big|\, t<-2 \cos \frac{\pi n }{n+1}\right\}, \quad & H_{+}&=\left\{t\in \re \,\big|\, t>-2 \cos \frac{\pi}{n+1}\right\}.&
 \end{align*}
 On the other hand, since for all $t$, $(\mathrm{d}/\mathrm{d} t)\mathrm{H}>\mathrm{O}$,
 \begin{align*}
I_{++}&=\bigcup_{k \in \mathbb{Z}} \big(-\pi+2 \pi k, 2 \pi k\big), \quad  & I_{+-}&=\bigcup_{k \in \mathbb{Z}} \big( 2 \pi k, \pi+ 2\pi k \big),&\\
 I_{--}&=\emptyset, \quad & I_{-+}&=\emptyset.&
 \end{align*}
 By Theorem \ref{final},  the eigenvalues of $\mathrm{A}^{-1}\mathrm{A}^*$ move strictly counterclockwise along $\mathbb{S}^1$ as $t$ increases on each of the intervals that make up the set  $(I_{++} \cap H_-) \cup (I_{+-} \cap H_+)$. Of course, these intervals depend on the order of $\mathrm{A}$. Fix $n=5$. Hence $H_-=(-\infty,-\sqrt{3})$ and $H_+=(\sqrt{3},\infty)$, which gives
 \begin{align*}
I_{++} \cap H_-&=  \bigcup_{k \in \mathbb{Z}\setminus\mathbb{N}} \big( -\pi+2 \pi k, 2 \pi k \big) \cup (-\pi, -\sqrt{3}),\\
I_{+-} \cap H_+&= (\sqrt{3},\pi) \cup \bigcup_{k \in \mathbb{N}\setminus\{0\}} \big( 2 \pi k, \pi+ 2 \pi k\big).
 \end{align*}
Table \ref{Table1} reports, for some values of $t$, the arguments of the eigenvalues of $\mathrm{A}^{-1}\mathrm{A}^*$ normalized to the interval $(0,2\pi]$. In the first column, we indicate the set to which the values of t belong. There the arguments of the eigenvalues increase as $t$ increases, in concordance with Theorem \ref{final}. The highlighted rows correspond to values of $t$ in intervals that do not belong to any of the sets described in Theorem \ref{final}. Only for some of these intervals the eigenvalues of $\mathrm{A}^{-1}\mathrm{A}^*$ are not monotone functions of $t$. This reminds us that indeed our conditions are only sufficient conditions. Figure \ref{fig1} shows the behavior of the arguments on $(2\pi, 7\pi)$. Although in the intervals $(3\pi, 4\pi)$ and $(5\pi, 6\pi)$ they are not monotone functions of $t$ (in the rest they are, again in concordance with Theorem \ref{final}), the prevailing direction of movement when $t$ increase is positive (see Figure \ref{fig2}). Actually, when $t$ tends to $\pm \infty$, they apparently converge and, in spite of the fact that they do not do it monotonically (see Figure \ref{fig2}), Theorem \ref{final} allows us to detect certain intervals in which this happens.
 \begin{center}\small \small \small
 \begin{table}[h]
\centering
\begin{tabular}{c c c c c c c}
\hline 
\rule{0pt}{1.2em} Sets & $t$ & $\theta_1$ & $\theta_2$ & $\theta_3$ & $\theta_4$ & $\theta_5$  \\
\hline
\rule{0pt}{1.2em}  $I_{++} \cap H_-$ & $-9$ & $0.185831$ & $0.335069$ & $0.578815$ & $0.86836$ & $1.11515$  \\
\rule{0pt}{1.2em} & $-7$ & $0.382408$ & $0.605763$ & $0.933046$ & $1.28985$ & $1.55583$  \\
\rule{0pt}{1.2em} & $-6.3$ & $0.435125$ & $0.687466$ & $1.04999$ & $1.4372$ & $1.7185$ \\
\hline \rowcolor[gray]{0.95}
\rule{0pt}{1.2em} & $-5$ & $0.444007$ & $0.725935$ & $1.16298$ & $1.63852$ & $1.98166$   \\
\rowcolor[gray]{0.95} \rule{0pt}{1.2em} & $-4.5$ & $0.418563$ & $0.709818$ & $1.18639$ & $1.71359$ & $2.0937$   \\
\rowcolor[gray]{0.95} \rule{0pt}{1.2em} & $-4$ & $0.388862$ & $0.697837$ & $1.22427$ & $1.8118$ & $2.22885$   \\
\hline
\rule{0pt}{1.2em}  $I_{++} \cap H_-$ & $-3$ & $0.402512$ & $0.783405$ & $1.44259$ & $2.14272$ & $2.59102$ \\
\rule{0pt}{1.2em} & $-2$ & $0.639983$ & $1.14671$ & $1.94338$ & $2.6444$ & $3.02789$ \\
\rule{0pt}{1.2em} &  $-1.9$ & $0.676005$ & $1.19954$ & $2.0053$ & $2.69728$ & $3.07083$   \\
\hline
\rowcolor[gray]{0.95} \rule{0pt}{1.2em} & $-1.5$ & $0.837019$ & $1.43245$ & $2.25954$ & $2.90327$ & $3.2362$   \\
\rowcolor[gray]{0.95} \rule{0pt}{1.2em}  & $0$ & $1.602$ & $2.41331$ & $3.14159$ & $3.56758$ & $3.77881$   \\
\rowcolor[gray]{0.95} \rule{0pt}{1.2em} & $1.6$ & $2.95322$ & $3.69139$ & $4.08756$ & $4.26589$ & $4.34323$  \\
\hline
\rule{0pt}{1.2em}  $I_{+-} \cap H_+$ & $1.8$ & $3.24452$ & $3.8919$ & $4.21492$ & $4.35617$ & $4.41543$ \\
\rule{0pt}{1.2em} & $2$ & $3.56526$ & $4.09244$ & $4.33892$ & $4.44364$ & $4.49258$ \\
\rule{0pt}{1.2em} &  $3$ & $4.72952$ & $4.75027$ & $4.78687$ & $4.86612$ & $5.14634$   \\
\hline
\rowcolor[gray]{0.95} \rule{0pt}{1.2em}  & $3.5$ & $4.83891$ & $4.88426$ & $4.95832$ & $5.08044$ & $5.3373$ \\
\rowcolor[gray]{0.95} \rule{0pt}{1.2em} & $4$ & $4.91632$ & $4.96699$ & $5.05207$ & $5.18974$ & $5.38445$ \\
\rowcolor[gray]{0.95} \rule{0pt}{1.2em} &  $5$ & $5.00418$ & $5.04361$ & $5.12066$ & $5.24861$ & $5.41438$   \\%
\hline
\end{tabular}
\medskip
\caption{Arguments of the eigenvalues of $\mathrm{A}^{-1}\mathrm{A}^*$.}\label{Table1}
\end{table}
\end{center}
\begin{figure}[h]
\centering
  \includegraphics[width=9cm]{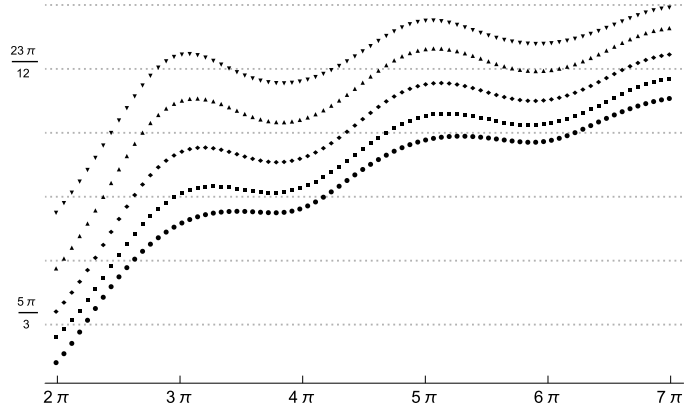}
  \caption{Arguments of the eigenvalues of $\mathrm{A}^{-1}\mathrm{A}^*$.}\label{fig1}
\end{figure}

\begin{figure}[h]
\centering
  \includegraphics[width=9cm]{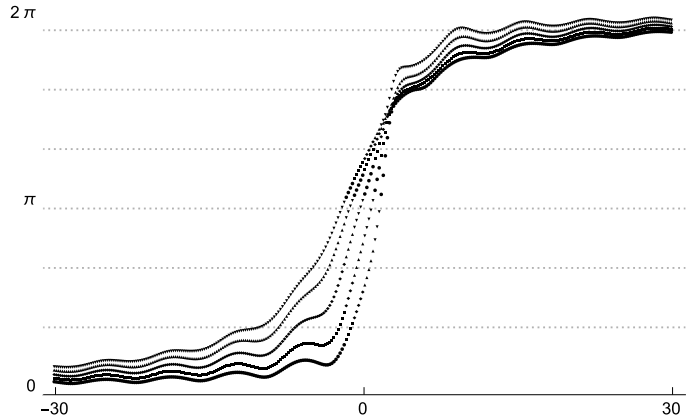}
  \caption{Arguments of the eigenvalues of $\mathrm{A}^{-1}\mathrm{A}^*$.}\label{fig2}
\end{figure}

\end{eje}

 A final question could be raised here: Is Theorem \ref{final} an unexpected result? The answer is no. Set $z:=x+i y$, $x$ and $y$ being differentiable real valued functions of a real variable, say $t$, with $y$ positive. Assume, for instance, that $x$ is negative, and $x$ and $y$ are strictly increasing functions of $t$. Since 
 $$
\theta:=\arg(z^{-1}\overline{z})=-2\arctan\left(\frac{y}{x} \right) \quad ({\rm mod}(0,2\pi]),
 $$ 
 $\theta$ moves strictly counterclockwise on $\mathbb{S}^1$ as $t$ increase. Roughly speaking, what we did in this work was to set up similar results in a matrix context.

\section*{Acknowledgements}
The author is supported by the Portuguese Government through the Funda\c{c}\~ao para a Ci\^encia e a Tecnologia (FCT) under the grant SFRH/BPD/101139/2014. This work is partially supported by the Centre for Mathematics of the University of Coimbra -- UID/MAT/00324/2013, funded by the Portuguese Government through FCT/MCTES and co-funded by the European Regional Development Fund through the Partnership Agreement PT2020. 
\bibliographystyle{plain}      

\bibliography{bib}   

\end{document}